\documentclass{amsart}
\usepackage{amsthm,amscd,amssymb,verbatim,epsf,amsmath,amsfonts,mathrsfs,graphicx}
\usepackage[colorlinks=true,linkcolor=blue,citecolor=blue]{hyperref}

\begin{document}
\theoremstyle{plain}
\newtheorem{Thm}{Theorem}
\newtheorem{Cor}{Corollary}
\newtheorem{Con}{Conjecture}
\newtheorem{Main}{Main Theorem}
\newtheorem{Lem}{Lemma}
\newtheorem{Prop}{Proposition}
\theoremstyle{definition}
\newtheorem{Def}{Definition}
\newtheorem{Note}{Note}
\newtheorem{Ex}{Example}
\theoremstyle{remark}
\newtheorem{notation}{Notation}
\renewcommand{\thenotation}{}
\errorcontextlines=0
\numberwithin{equation}{section}
\renewcommand{\rm}{\normalshape}%

\title[Fredholm-Regularity of Holomorphic Discs in Plane Bundles]%
   {Fredholm-Regularity of Holomorphic Discs in Plane Bundles over Compact Surfaces}
\author{Brendan Guilfoyle}
\address{Brendan Guilfoyle\\
          School of Science, Technology, Engineering and Mathematics\\
          Institute of Technology, Tralee \\
          Clash \\
          Tralee  \\
          Co. Kerry \\
          Ireland.}
\email{brendan.guilfoyle@ittralee.ie}
\author{Wilhelm Klingenberg}
\address{Wilhelm Klingenberg\\
 Department of Mathematical Sciences\\
 University of Durham\\
 Durham DH1 3LE\\
 United Kingdom}
\email{wilhelm.klingenberg@durham.ac.uk }

\begin{abstract}
We study the space of holomorphic discs with boundary on a surface in a real 2-dimensional vector bundle over a compact 2-manifold. We prove that, if the ambient 4-manifold admits a fibre-preserving transitive holomorphic action, then a section with a single complex point has $C^{2,\alpha}$-close sections such that any (non-multiply covered) holomorphic disc with boundary in these sections are Fredholm regular. 

Fredholm regularity is also established when the complex surface is neutral K\"ahler, the action is both holomorphic and symplectic, and the section is Lagrangian with a single complex point.

\vspace{0.1in}

On etude l'espace des courbes holomorphes a bord dans une surface reelle situe dans une fibree vectoriel de rang 2 sur une variete reelle a dimension deux. On prouve que, si le fibree ambient admet une action transitive et holomorphe qui preserve la fibration, alors une section avec un et seulement un point complexe admet des deformation petits dans la norme $C^{2,\alpha}$ tel que toute disque holomorphe a bord dans la deformation soit Fredholm reguliere. 

On prouve aussi la Fredholm regularite dans le cas que le fibree ambient est Kaehlerien a signature (2,2), l'action de la groupe e holomorphe et symplectique, et la surface bordante est Lagrangienne avec un seul poit complexe.

\end{abstract}

\maketitle

The problem of describing the set of holomorphic curves in symplectic 4-manifolds has generated a wealth of results over the past three decades \cite{AuaLaf} \cite{cam} \cite{gromov} \cite{hls} \cite{mcDaS} \cite{pol}. The techniques developed have led to non-existence theorems for Lagrangian submanifolds, as well as applications of fillings by holomorphic discs, insights into the structure of symplectic 4-manifolds and the emergence of Floer homology.

Of particular interest is the case of holomorphic discs with boundary lying on a Lagrangian surface. In this situation, it is usually assumed that the symplectic form on the ambient 4-manifold {\it tames} the (almost) complex structure, in that their composition yields a positive definite metric. In this case, Lagrangian surfaces are totally real and therefore are also good boundary conditions for holomorphic discs.

In this paper we consider the situation of minimal deviation from the usual setting: a real surface in $P$ with a single complex point, namely a point $p$ on the surface whose tangent plane is a complex line in $T_pP$. In general, such a boundary condition would not be Fredholm regular, however, with some extra assumptions on the ambient complex surface, we show that genericity can still be extracted. 

In particular, let $P\rightarrow M$ be a real 2-plane bundle over a compact surface $M$ without boundary. Let ${\mathbb J}$ be a complex structure on the total space $P$, and assume that $(P,{\mathbb J})$ admits a fibre-preserving group $G$ of holomorphic automorphisms that are transitive on the fibres. 

Examples of such manifolds include complex line bundles over $M$ for which the first Chern class is greater than or equal to the genus of $M$. In this case, the group is the space of global holomorphic sections, whose existence is guaranteed by the Riemann-Roch Theorem, and the action is addition in the vector space of sections by elements of the group. Other examples include certain spaces of oriented geodesics on manifolds with isometric group action \cite{gak}.

The main results we prove are:
\vspace{0.1in}

\begin{Thm} \label{t:1}
Let $\Sigma_0\subset P$ be a $C^{2+\alpha}$-smooth global section for $0<\alpha<1$ with one and only one complex point occurring at $p\in \Sigma_0 \subset P$. Denote by ${\mathcal U}$ be the collection of all those $C^{2+\alpha}$-close sections that pass through $p$.

Then there exists an open dense subset ${\mathcal W}\subset{\mathcal U}$ such that any (non-multiply covered) holomorphic disc with boundary on $\Sigma\in {\mathcal W}$ is Fredholm regular.
\end{Thm}
\vspace{0.1in}

In the K\"ahler case, where $(P,{\mathbb J})$ has a compatible symplectic structure $\Omega$, one can also seek holomorphic discs whose boundary lie on a Lagrangian surface which contains a single complex point. 

As mentioned above, if the symplectic form tames the complex structure, a Lagrangian surface is automatically totally real and therefore cannot contain complex points. Thus, in order to consider Lagrangian surface with complex points, we assume that the compatible metric is indefinite of signature (2,2). Such metrics arise, for example, on spaces of geodesics \cite{gak4}.

Suppose $(P,{\mathbb J},\Omega)$ is neutral K\"ahler, i.e. the complex structure and symplectic form are compatible and the associated metric $G(\cdot,\cdot)=\Omega({\mathbb J}\cdot,\cdot)$ is of signature (2,2) - and the transitive action is both holomorphic and symplectic, then we prove that:

\vspace{0.1in}
\begin{Thm}\label{t:2}
Let $\Sigma_0\subset P$ be a $C^{2+\alpha}$-smooth Lagrangian section with just one complex point and ${\mathcal U}$ be the space of Lagrangian sections that are $C^{2+\alpha}$-close to $\Sigma_0$.

Then, there exists an open dense subset ${\mathcal W}\subset{\mathcal U}$ such that any (non-multiply covered) holomorphic disc with boundary on $\Sigma\in {\mathcal W}$ is Fredholm regular.

\end{Thm}
\vspace{0.1in}

A special case of Theorem \ref{t:2} arises in the following setting studied in \cite{gak4}. Consider the collection ${\mathbb L}({\mathbb R}^3)$ of oriented lines of Euclidean 3-space, which may be identified with the total space of the tangent bundle to the 2-sphere. This takes the role of the plane bundle $P$ and it admits a canonical metric of signature $(2,2)$. 

This metric is K\"ahler, with compatible complex and symplectic structures, and is scalar flat, although it is not Einstein . The Euclidean group gives a  fibre-preserving group of symplectic, holomorphic automorphisms that are transitive on the fibres of $P$. 

An oriented smooth surface in ${\mathbb R}^3$ gives rise, through its oriented normal lines, to a smooth surface in ${\mathbb L}({\mathbb R}^3)$. This surface is Lagrangian and the induced metric is either Lorentz or degenerate, where the degeneracy occurs precisely at the umbilic points of the surface. These degeneracies correspond to complex points on the Lagrangian surface in  ${\mathbb L}({\mathbb R}^3)$.

The Carath\'eodory Conjecture, first stated in the 1920's, claims that the number of isolated umbilic points on a closed convex surface needs to be at least two. In a reformulation in terms of its normals in ${\mathbb L}({\mathbb R}^3)$, this leads to the same lower bound on the number of complex points on a Lagrangian sections of the tangent bundle of the 2-sphere. 

A key element of the authors' approach to this conjecture is the Fredholm regularity for holomorphic discs with Lagrangian boundary in this setting. Theorems \ref{t:1} and \ref{t:2} establish Fredholm regularity in the generalized setting of neutral K\"ahler plane bundles with isometric group action, and in particular, verify this part of the unpublished proof of 2013 \cite{gak2}.

In the next section we consider real surfaces in a complex surface and prove Theorem \ref{t:1}. Following a brief discussion of the difference between the definite and indefinite K\"ahler cases, we consider Lagrangian surfaces in an indefinite K\"ahler 4-manifold and prove Theorem \ref{t:2}.

\vspace{0.1in}

\section{Real Boundary Conditions}

Let $\pi:P\rightarrow M$ be a real 2-plane bundle over a surface $M$ endowed with a complex structure ${\mathbb J}$. Assume that $(P,{\mathbb J})$ admits a fibre-preserving group $G$ of holomorphic diffeomorphisms whose restriction to any fibre of $P$ is transitive. 

We now analyze holomorphic discs in the complex surface $(P,{\mathbb J})$ with boundary lying on a surface $\Sigma$. Fix  $\alpha\in(0,1)$, $s\geq 1$ and denote by $C^{k+\alpha}$ and $H^{1+s}$ the usual H\"older and Sobolev spaces, respectively.

Define the space of {\it H\"older boundary conditions} by
\[
{\mathcal S}\equiv \left\{ \;\Sigma\subset P \; \;\left| \;\;  \Sigma\; {\mbox{is a}}\; C^{2+\alpha} \;{\mbox{- \;section of }}P\rightarrow M\;\right.\right\}  .
\]
Endow ${\mathcal S}$ with the $C^{2+\alpha}$ topology  for global sections of $P$. Since it preserves fibres, the group $G$ acts continuously on ${\mathcal S}$ by composition with the section. 

For a fixed $\Sigma\in{\mathcal S}$, the differentiable structure of the set of $C^{2+\alpha}$ sections of the normal bundle $N\Sigma=T_\Sigma P/T\Sigma$ exponentiates to give an infinite-dimensional manifold structure for an open neighbourhood of $\Sigma$ in ${\mathcal S}$. 

If $\Sigma$ were totally real, then $N\Sigma\cong {\mathbb J}T\Sigma\cong T\Sigma$ where these are identified as sub-bundles of $T_{\Sigma}P$. This gives a canonical Banach manifold structure modeled on sections of $T\Sigma$:
\[
\Gamma( T \Sigma) \cong \Gamma (J T\Sigma)  \hookrightarrow \Gamma( N \Sigma) \stackrel{{\mbox{exp}}}{\rightarrow} {\mathcal S} .
\] 

In our situation, the section will be allowed to have a complex point and so we must modify this argument. 

In particular, for a fixed point $ p\in P$ 
define
\[
{\mathcal S}_0\equiv \left\{ \;\Sigma\in{\mathcal S}\;\; \left| \;\;   p\in\Sigma \;\right.\right\}  ,
\]
endowed with the induced  H\"older space $C^{2+\alpha}$ -topology.

Note that ${\mathcal S}_0$ can be identified with the quotient of ${\mathcal S}$ by $G_0=G/{\mbox{stab}}(p)$ since for every $\Sigma \in \mathcal S$ there exists a unique element of $G_0$ taking $\pi^{-1}(\pi( p))\cap\Sigma\in{\mathcal S}$ to $p$. In other words, we have a homeomorphism ${\mathcal S}_0 \to {\mathcal S}/G_0$, where the equivalence relation is $\Sigma_1\sim \Sigma_2$ iff there exists $g\in G_0$ such that $g(\Sigma_1)=\Sigma_2$.

Let $\Sigma_0\in{\mathcal S}_0$ have one and only one complex point and assume it occurs at $p\in\Sigma_0\subset P$. Then there exists an open neighbourhood of $\Sigma_0$ in
${\mathcal S}_0$ which is a Banach manifold. We prove this as follows.

\vspace{0.1in}

\begin{Lem}\label{l:1}
Let $\Sigma_0 \in {\mathcal S}_0 $ be a surface with one and only one complex point occuring at $ p$.
Then there exists a local homeomorphism
\[
\Phi : \Gamma ({\mathbb J}T\Sigma_0){\mbox{ mod }}\Gamma (T\Sigma_0)   \hookrightarrow {\mathcal S}_0 , 
\]
where we view $T\Sigma_0$  and ${\mathbb J}T\Sigma_0$ as vector sub-bundles of $T_{\Sigma_0}P$, and $s_1 \sim s_2$  if there exists a section $t \in \Gamma (T\Sigma_0)$ depending on $s_1 , s_2 $ with $s_1 - s_2 = t $.

In particular,  there exists an open neighbourhood of  $\Sigma_0$ in ${\mathcal S}_0$ which has a canonical Banach manifold structure.
\end{Lem}

\begin{proof}
For any surface $\Sigma\subset P$ let $T_{\Sigma} P{\mbox{ mod }}T\Sigma\cong  N\Sigma$ be the normal bundle of $\Sigma$. In our case $\Sigma_0$ is a section, and so we can identify $N\Sigma_0\cong{\mbox{ker }}d\pi|_{\Sigma_0}$.

Now, the key point is that if $\Sigma_0$ has only one complex point at $p$, then ${\mathbb J}T\Sigma_0{\mbox{ mod }}T\Sigma_0$ is a plane bundle with vanishing fibre at the complex point, where ${\mathbb J}T_{ p} \Sigma_0 = T_{ p}\Sigma_0$. Thus
\[
{\mathbb J}T\Sigma_0{\mbox{ mod }}T\Sigma_0|_{\Sigma_0-\{p\}} \cong N\Sigma_0|_{\Sigma_0-\{p\}}
\]
is an isomorphism of plane bundles over $\Sigma_0-\{p\}$. 

We now use the transitive action of $G$ to identify sections of these bundles over the whole of $\Sigma_0$. $G_0$ maps sections 
$\Gamma(N{\Sigma_0})\cong\Gamma({\mbox{ker }}d\pi|_{\Sigma_0})$ to sections. Define
\[
\Gamma_0(N\Sigma_0)=\{s:M\rightarrow N\Sigma_0\;\left|\; s{\mbox{ is a}}\; C^{2+\alpha}{\mbox{-section with}}\;\;s(\pi(p))=0\;\right.\}
\]
Then $\Gamma(N{\Sigma_0})/G_0 \cong \Gamma_0(N\Sigma_0)$, as we can take sections that vanish at $p$ as representatives of the equivalence classes on the left.

Moreover, we can identify $\Gamma({\mathbb J}T\Sigma_0{\mbox{ mod }}T\Sigma_0)$ with $\Gamma_0( N \Sigma_0)$ and 
\[
 \Gamma_0( N \Sigma_0) \stackrel{{\mbox{exp}}}{\rightarrow} {\mathcal S}_0 
\]
is a global smooth  embedding. Thus the composition
\[
\Phi : \Gamma ({\mathbb J}T\Sigma_0){\mbox{ mod }}\Gamma (T\Sigma_0)
 \cong \Gamma_0(N \Sigma_0)
 \stackrel{{\mbox{exp}}}{\rightarrow} {\mathcal S}_0
\] 
is a canonical embedding of the Banach space of $C^{2+\alpha} $ - smooth sections of  ${\mathbb J}T\Sigma_0{\mbox{ mod }}T\Sigma_0$
into $ {\mathcal S}_0$, giving rise to an open Banach manifold of variations of $ \Sigma_0 \in {\mathcal S}_0$.

\end{proof}
\vspace{0.1in}

For $s > 1, (2 - \alpha)/2 = 1/s $ and a relative class $A\in\pi_2(P,\Sigma)$, the space of
{\it parameterized Sobolev-regular discs with Lagrangian boundary condition} is defined by
\[
{\mathcal{F}}_A\equiv \left\{\;(f, \Sigma)\in H^{1+s}(D,P)\times{\mathcal U}\;\;\left|\;\; [f]=A,\; f(\partial D)\subset {\mbox{ totally real part of }} \Sigma   
\;\right.\right\},
\]
where ${\mathcal U}$ is the Banach manifold neighbourhood of $\Sigma_0$ as above.
The space ${\mathcal{F}}_A$ is a Banach manifold and the the projection 
$\pi:{\mathcal{F}}_A\rightarrow {\mathcal{U}}$: $\pi(f, \Sigma)= \Sigma$ is a Banach
bundle. 

For $(f, \Sigma)\in{\mathcal F}_A$ define 
$\bar{\partial}f={\textstyle{\frac{1}{2}}}(df\circ j- J \circ df)$, where
$j$ is the complex structure on $D$. Then
$\bar{\partial}f\in H^s(f^*T ^{01}P) \equiv H^s(f^*TP) $ and we define the space of sections
\[
H^s\equiv\bigcup_{(f, \Sigma)\in{\mathcal F}_A}H^s(f^*TP).
\]
This is a Banach vector bundle over ${\mathcal F}_A$ and the operator $\bar{\partial}$ is
a section of this bundle.

\vspace{0.1in}
\begin{Def}
The {\it set of holomorphic discs with H\"older boundary condition} is defined by
\[
{\mathcal M}_A\equiv\left\{\;(f, \Sigma)\in{\mathcal F}_A\;\;\left|\;\; \bar{\partial}f=0 \;\right.\right\}.
\]
\end{Def}
\vspace{0.1in}

As before let $\Sigma_0$ be a H\"older section with a single complex point at $ p$. We now prove our main result.

\vspace{0.1in}
\noindent{\bf Theorem \ref{t:1}}:
{\it 
Let $\Sigma_0\subset P$ be a $C^{2+\alpha}$-smooth section for $0<\alpha<1$ with just one complex point at $p\in P$, and let ${\mathcal U}$ be the space of 
$C^{2+\alpha}$-close sections that pass through $p$.

Then, there exists an open dense subset ${\mathcal W}\subset{\mathcal U}$ such that any (non-multiply covered) holomorphic disc with boundary on 
$\Sigma\in {\mathcal W}$ is Fredholm regular.
}
\vspace{0.1in}

\begin{proof}
We first show that there exists a neighbourhood of $\Sigma_0$, denoted
${\mathcal V}\subset{\mathcal S}_0$, such that ${\mathcal M}_A\cap\pi^{-1}({\mathcal V})$ is a Banach submanifold of ${\mathcal F}_A\cap\pi^{-1}({\mathcal V})$.
  
This is established by considering the smooth map
\[
\Delta:{\mathcal F}_A\cap\pi^{-1}({\mathcal U})\rightarrow H^s\times \Omega(P),
\]
defined by $\Delta(f,\Sigma)=(\bar{\partial}f,\Phi^{-1}_\Sigma\circ f|_{\partial D})$, where 
$\Omega(P)$ is the set of loops in $P$ and $\Phi_\Sigma$ is an ambient smooth isotopy which takes $\Sigma_0$ to $\Sigma\in{\mathcal U}$. For ease
of notation, we suppress the composition $\Phi^{-1}_\Sigma\circ f|_{\partial D}$ and simply write $f|_{\partial D}$.

Thus
\[
{\mathcal M}_A\cap\pi^{-1}({\mathcal U})=\Delta^{-1}(\{0\}\times\Omega_A({\mathcal U})).
\]
Here
\[
\Omega_A({\mathcal U})=\bigcup_{\Sigma\in{\mathcal U}}\{{\mbox{ loops in }}\Sigma{\mbox{ bounding an embedded disc in the class }}A\in\pi_2(P,\Sigma)\}.
\]
To prove the Theorem by the implicit function theorem we must show that $\Delta$ is transverse to the submanifold
\[
\{0\}\times\Omega_A({\mathcal U})\subset H^s\times \Omega(P),
\]
at $\Sigma_0$. Transversality at $\Sigma_0$ is proved by a modification of Theorem \ref{t:1} in Oh \cite{oh}, which we now outline.

Let $(f,\Sigma)\in{\mathcal M}_A$, which means that
\[
\bar{\partial}f=0
\qquad\qquad
f|_{\partial D}\subset\Sigma_0
\qquad\qquad
[f]=A {\mbox{ in }}\pi_2(P,\Sigma_0).
\]
To show transversality, we need to prove that
\[
{\mbox{Im}}\left(D_{(f,\Sigma)}\Delta\right)+\{0\}\oplus T_{f|_{\partial D}}\Omega_A({\mathcal U})=
T_{(0,f|_{\partial D})} H^s\times \Omega(P),
\]
or, denoting the $L^2$-adjoint by $^{\perp_{L^2}}$, equivalently
\begin{align}
0&=\left({\mbox{Im}}\left(D_{(f,\Sigma)}\Delta\right)+\{0\}\oplus T_{f|_{\partial D}}\Omega_A({\mathcal U})\right)^{\perp_{L^2}}\nonumber\\
&=\left({\mbox{Im}}\left(D_{(f,\Sigma)}\Delta\right)\right)^{\perp_{L^2}}\cap\left(\{0\}\oplus T_{f|_{\partial D}}\Omega_A({\mathcal U})\right)^{\perp_{L^2}}\nonumber\\
&=\left({\mbox{Im}}\left(D_{(f,\Sigma)}\Delta\right)\right)^{\perp_{L^2}}\cap\left(H^{-s}\oplus \left(T_{f|_{\partial D}}\Omega_A({\mathcal U})\right)^{\perp_{L^2}}\right)\label{e:adj}.
\end{align}
A point in $T_{(f,\Sigma_0)}\left({\mathcal F}_A\cap\pi^{-1}({\mathcal U})\right)$ can be represented by $(\zeta,X_h)$, where 
$X_h={\mathbb J}({\mbox{grad}}_gh){\mbox{ mod }}T\Sigma_0$ is the normal vector field associated 
with some $h\in C_0^\infty(\Sigma_0)$ and an ambient background Riemannian metric $g$ on $P$, and we find that
\[
D_{(f,\Sigma_0)}\Delta(\zeta,X_h)=({\nabla}^+_{\mathbb J}\zeta,X_h(f|_{\partial D})-\zeta(f|_{\partial D})),
\]
where we have introduced the connection
\[
\nabla^{\pm}_{\mathbb J}={\textstyle{\frac{1}{2}}}\left(\frac{D}{dx}\pm{\mathbb J}\frac{D}{dy}\right),
\]
where $D$ is the connection on $(P,g)$.

Now let $(\psi,\alpha)\in\left({\mbox{Im}}\left(D_{(f,\Sigma)}\Delta\right)+\{0\}\oplus T_{f|_{\partial D}}\Omega_A({\mathcal U})\right)^{\perp_{L^2}}$. We show that  $(\psi,\alpha)=(0,0)$ as follows. By definition we have 
\[
\int_D({\nabla}^+_{\mathbb J}\zeta,\psi)+\int_{\partial D}(X_h(f|_{\partial D})-\zeta(f|_{\partial D}),\alpha)=0,
\]
for all $(\zeta,X_h)\in T_{(f,\Sigma_0)}\left({\mathcal F}_A\cap\pi^{-1}({\mathcal U})\right)$.

Integrating by parts and rearranging terms
\[
-\int_D(\zeta,{\nabla}^-_{\mathbb J}\psi)+\int_{\partial D}(\zeta,\tilde{\psi}-\alpha)+\int_{\partial D}(X_h\circ f|_{\partial D}),\alpha)=0,
\]
where $\tilde{\psi}$ is the $({\mathbb J},g)$-adjoint of $\psi$. As this holds for all $\zeta$ and $h$, we have
\[
{\nabla}^-_{\mathbb J}\psi=0
\qquad\qquad
\tilde{\psi}-\alpha=0 {\mbox{ on }}\partial D
\qquad\qquad
\alpha^\perp=0 {\mbox{ on }}\partial D .
\]
Since, by equation (\ref{e:adj}) $\alpha\in\left(T_{f|_{\partial D}}\Omega_A({\mathcal U})\right)^{\perp_{L^2}}$, we have $\alpha=\alpha^\perp$, which vanishes by the last equation above. Substitute this in the second equation to get $\tilde{\psi}=0$ on $\partial D$, and finally, by ellipticty and uniqueness, from the first equation with this boundary condition, $\psi=0$. Thus $(\psi,\alpha)=(0,0)$, as claimed, and transversality is established.

Consider the linearization of $\bar{\partial}$ at $(f, \Sigma) \in \mathcal{F}_A$ with respect to any $J$-parallel connection on $H^{1+s}({\mathcal F}_A)$:
\[
\nabla_{(f, \Sigma)} \bar{\partial}: H^{1+s}(f^*TP \otimes f^*T \Sigma) \to H^s(f^*TP).
\]
First note that this differential operator has the same symbol as the Cauchy-Riemann operator, and is elliptic. Now by Theorem 19.5.1 and Theorem 19.5.2 in \cite{hoermander} the operator $\Delta$ is Fredholm and has finite dimensional kernel and co-kernel. 

Moreover, by the Sard-Smale Theorem for Fredholm operators (see e.g. Theorem 1.3 of \cite{smale}), there exists a dense open set ${\mathcal W}\subset{\mathcal U}$ such that any (not multiply covered) holomorphic disc with boundary in $\Sigma\in{\mathcal W}$ is Fredholm-regular i.e. ${\mbox{dim coker}}(\nabla_{(f, \Sigma)} \bar{\partial})=0$.
\end{proof}
\vspace{0.1in}

\section{Lagrangian Boundary Conditions}

In the usual case for which Fredholm regularity of holomorphic discs is established, either totally real or Lagrangian boundary conditions can be implemented \cite{oh}. This is because it is assumed that the symplectic form on the ambient 4-manifold {\it tames} the (almost) complex structure, in that their composition yields a positive definite metric and therefore Lagrangian surfaces are totally real. 

In order to prove the above result for Lagrangian surfaces with a single complex point and the nearby Lagrangian surfaces, one must therefore consider the non-tame case: when the symplectic form and complex structure combine to form an indefinite metric. The signature of
such a metric is necessarily $(2,2)$, which we refer to as {\it neutral} K\"ahler. 

We now briefly discuss the difference between neutral and definite K\"ahler structures on a 4-manifold, particularly as regards co-dimension 2 sub-manifolds.

Let (${\mathbb{M}},{\mathbb{G}},{\mathbb{J}},\Omega$) be a K\"ahler surface. That is, ${\mathbb{M}}$ is a real 4-manifold endowed with the following structures. First, there is the metric ${\mathbb{G}}$, which we do not insist be positive definite - it may also have neutral signature (2,2). In order to deal with both cases simultaneously we assume that the metric can be diagonalized pointwise to ($1,1,\epsilon,\epsilon$), for $\epsilon=\pm1$.

In addition, we have a complex structure ${\mathbb{J}}$, which is a mapping
${\mathbb{J}}:$T$_p{\mathbb{M}}\rightarrow $T$_p{\mathbb{M}}$ at each $p\in{\mathbb{M}}$, which satisfies ${\mathbb{J}}^2=-{\mbox{Id}}$ and 
satisfies certain integrability conditions (the vanishing of the Nijenhuis tensor). 
Finally, there is a symplectic form $\Omega$, which is a closed non-degenerate 2-form. These structures are required  
to satisfy the compatibility conditions:
\[
{\mathbb{G}}({\mathbb{J}}\cdot,{\mathbb{J}}\cdot)={\mathbb{G}}(\cdot,\cdot) \qquad\qquad
{\mathbb{G}}(\cdot,\cdot)=\Omega({\mathbb{J}}\cdot,\cdot).
\]

The following result highlights the difference between the case where ${\mathbb{G}}$ is positive definite and where 
it is neutral:
\vspace{0.1in}
\begin{Prop}\cite{gak}
Let $p\in{\mathbb{M}}$ and $v_1,v_2\in T_p{\mathbb{M}}$ span a plane. Then
\[
\Omega(v_1,v_2)^2+\epsilon\varsigma^2(v_1,v_2)=|{\mbox{det }}{\mathbb{G}}(v_i,v_j)|,
\]
where $\varsigma^2(v_1,v_2)\geq0$ with equality iff $\{v_1,v_2\}$ spans a complex plane.
\end{Prop}
\vspace{0.1in}

As a consequence, in the definite case, a Lagrangian plane, for which $\Omega(v_1,v_2)=0$, cannot be complex since this would mean that $\varsigma^2(v_1,v_2)=0$ and hence the metric would be degenerate, which is impossible.

In the neutral case, however, the metric induced on a plane can be degenerate and hence a plane can be both Lagrangian and holomorphic, as we illustrate in the next example.

\vspace{0.1in}
\begin{Ex}
Consider the neutral K\"ahler Structure on ${\mathbb R}^4$. That is, let $(x^1,x^2,x^3,x^4)$ be standard coordinates on ${\mathbb R}^4$ and consider 
the flat neutral metric $g$ defined
\[
ds^2=(dx^1)^2+(dx^2)^2-(dx^3)^2-(dx^4)^2.
\]
The complex structure
\[
J\left(X^1\frac{\partial}{\partial x^1}+X^2\frac{\partial}{\partial x^2}+X^3\frac{\partial}{\partial x^3}+X^4\frac{\partial}{\partial x^4}\right)
=-X^2\frac{\partial}{\partial x^1}+X^1\frac{\partial}{\partial x^2}-X^4\frac{\partial}{\partial x^3}+X^3\frac{\partial}{\partial x^4},
\]
is compatibility with the flat metric and has associated symplectic form given by
\[
\Omega=dx^1\wedge dx^2-dx^3\wedge dx^4 
\]
Thus $({\mathbb R}^4,g,J,\Omega)$ is a neutral K\"ahler structure.

Define the plane $\Pi$ by
\[
\Pi=a\left(\frac{\partial}{\partial x^1}+\frac{\partial}{\partial x^3}\right)+b\left(\frac{\partial}{\partial x^2}+\frac{\partial}{\partial x^4}\right)
\]
for $a,b\in{\mathbb R}$. A quick check shows that $\Pi$ is both Lagrangian and holomorphic:
\[
\Omega(\Pi)=0  \qquad\qquad J(\Pi)=\Pi
\]
\end{Ex}

\vspace{0.1in}
The space of Lagrangian boundary conditions in the plane bundle $P$ which we consider is 
\[
{\mathcal L}{\mbox{ag}}\equiv \left\{ \;\Sigma\subset P \; \;\left| \;\;  \Sigma\; {\mbox{is a}}\; C^{2+\alpha} \;{\mbox{ Lagrangian section of}} \; P\rightarrow M \;\right.\right\},  
\] 
and for a fixed $p \in P$ we identify those Lagrangian surfaces containing $p$ as
\[
{\mathcal L}{\mbox{ag}}_0\equiv \left\{ \;\Sigma\in{\mathcal L}{\mbox{ag}}\;\; \left| \;\;   p\in\Sigma \;\right.\right\}\equiv {\mathcal L}{\mbox{ag}}/ G_0,
\]
where $G$ is the fibre-preserving group of isometric automorphisms which is transitive on the fibres and $G_0=G/{\mbox{stab}}(p)$. Here and throughout we assume the existence of such a group which acts both holomorphically and symplectically. The collection of equivalence classes ${\mathcal L}{\mbox{ag}}_0$ is equipped with the quotient topology. We proceed as before with a Lagrangian surface containing a single complex point at $p \in P$.
\vspace{0.1in}
\begin{Lem}
Let $\Sigma_0 \in {\mathcal L}{\mbox{ag}}_0 $ be a (an equivalence class of)  surface with exactly one complex point at $p$. Then there exists an open neighbourhood ${\mathcal U}_{L}$ of  $\Sigma_0$ in ${\mathcal L}{\mbox{ag}}_0$ which has a canonical Banach manifold structure.
\end{Lem}
\begin{proof}
The Weinstein Tubular Neighborhood Theorem \cite{wein} says that a neighborhood of $\Sigma_0 \subset P$ is symplectically isomorphic to the canonical $T^*\Sigma_0$. In fact, the Lagrangian surfaces near $\Sigma_0$ are modeled by the Banach space of Lagrangian sections, for which $d({\mathbb J}(v)\lrcorner\Omega)=0$. Thus, define
\[
\Gamma^{lag}(N\Sigma_0)=\left\{\;v\in\Gamma(N\Sigma_0)\;\;\left|\;\; v\in C^{2+\alpha},\;\;d({\mathbb J}(v)\lrcorner\Omega)=0\;\right.\right\},
\] 
Away from complex points, $N\Sigma_0\equiv{\mathbb J}T\Sigma_0$, since the symplectic and complex structures are compatible. 

At a complex point,
${\mathbb J}T\Sigma_0\equiv T\Sigma_0$ and so ${\mathbb J}T\Sigma_0{\mbox{ mod }}T\Sigma_0=\{0\}$. If there is only one complex point at $p$, we model the space of Lagrangian variations of $\Sigma_0$ by the Lagrangian sections that vanish at $p$, as before. 

Thus, define
\[
\Gamma_0^{lag}(N\Sigma_0)=\{v\in\Gamma^{lag}(N\Sigma_0)\;\left|\; v(\pi(p))=0\;\right.\}.
\]

Exactly as in Lemma \ref{l:1}, but now carrying the Lagrangian condition, it follows that the Banach space 
\[
\Gamma_0^{lag}(N\Sigma_0)\equiv\Gamma^{lag} ({\mathbb J}T\Sigma_0{\mbox{ mod }}T\Sigma_0),
\]
models the open Banach manifold of variations of $ \Sigma_0$ in the space of Lagrangian surfaces passing through $p$.
\end{proof}

\vspace{0.1in}

If $(P,{\mathbb J},\Omega)$ is neutral K\"ahler and the transitive action $G$ is both holomorphic and symplectic, the following holds:

\vspace{0.1in}
\noindent{\bf Theorem \ref{t:2}}:
{\it 
Let $\Sigma_0\subset P$ be a $C^{2+\alpha}$-smooth Lagrangian section with just one complex 
point and ${\mathcal U}_{L}$ be the Banach manifold of Lagrangian sections that are $C^{2+\alpha}$-close to $\Sigma_0$.

Then there exists an open dense subset ${\mathcal W_{L}}\subset{\mathcal U_{L}}$ such that any (non-multiply covered) holomorphic disc with boundary on $\Sigma\in {\mathcal W_{L}}$ is Fredholm regular.
  
}
\vspace{0.1in}
\begin{proof}
We consider ${\mathcal{F}}_A$ to be as in the first section, with  ${\mathcal U}$ is replaced by Lagrangian boundary conditions ${\mathcal U_{L}}$. That is, we consider the map
$\Delta:{\mathcal F}_A\cap\pi^{-1}({\mathcal U_{L}})\rightarrow H^s\times \Omega(P)$ and 
\[
{\mathcal M}_A\cap\pi^{-1}({\mathcal U_{L}}) = \Delta^{-1}(\{0\}\times\Omega_A({\mathcal U_{L}})),
\]
where $ {\mathcal M}_A$ denotes the collection of holomorphic discs with boundary class in $ \Omega_A({\mathcal U_{L}})$. 

Now ${\mathcal M}_A\cap\pi^{-1}({\mathcal U_{L}})$ is shown to be a Banach manifold by transversality as in Theorem  1, where the argument is not affected by the Lagrangian property of the boundary condition. 

The application of the Theorem of Sard-Smale to the Fredholm operator  $\Delta$ then serves to complete the proof.
\end{proof}

\end{document}